\documentclass[reqno,b5paper]{amsart}
\usepackage{multicol}
\usepackage{amsmath}
\usepackage{amssymb}
\usepackage{amsthm}
\usepackage{threeparttable}
\usepackage{}
\usepackage{enumerate}
\usepackage[mathscr]{eucal}
\swapnumbers
\theoremstyle{plain}
\newtheorem{thm}{Theorem}[section]

\newtheorem{lem}[thm]{Lemma}
\newtheorem{Example}{Example}[section]
\newtheorem{note}{Note}[section]

\theoremstyle{definition}
\newtheorem{defn}{Definition}[section]
\newtheorem{rem}{Remark}[section]

\begin{document}

\setcounter {page}{1}
\title{rough $I$-statistical convergence of double sequences}

\author{ Prasanta Malik$^{1*}$ and Argha Ghosh$^{2}$\ }
\thanks{1 Department of Mathematics, University of Burdwan, Golapbag, Burdwan-713104, West Bengal, India. Email: pmjupm@yahoo.co.in}
\thanks{* Corresponding Author}
\thanks{2 Department of Mathematics, University of Burdwan, Golapbag, Burdwan-713104, West Bengal, India. Email: buagbu@yahoo.co.in}
\newcommand{\acr}{\newline\indent}
\maketitle

\begin{center}
\textbf{Abstract
}\end{center}

The notion of $I$-statistical convergence of a double sequence was first introduced by Belen et. al.\cite{Be}. 
In this paper we introduce and study the notion of rough $I$-statistical convergence
of double sequences in normed linear spaces. We also introduce the notion of rough $I$-statistical limit set of a double sequence
and discuss about some topological properties of this set.\\\\
\maketitle

\section{\textbf{Introduction:}} 
 The notion of convergence of real double sequences was first introduced by Pringsheim \cite{Pr}. A double 
sequence $x = \{x_{jk}\}_{ j, k \in \mathbb {N}} $ of real numbers is said to converge to a real 
number $l$, if for any $\varepsilon > 0 $, there exists $ m \in \mathbb{N} $ such that for all $ j, k \geq m $
\begin{center}
$| x_{jk} - l | < \varepsilon $ . 
\end{center} 
In this case we write $ \lim\limits_{\stackrel{\stackrel{j\rightarrow
\infty}{k\rightarrow \infty}} ~} x_{jk} = l  $. This notion of convergence 
of real double sequences has been extended to statistical convergence by Mursaleen et. al. \cite{Mu} 
( also by Moricz \cite{Mo} who introduced it for multiple sequences) using double natural density of $\mathbb{N} \times \mathbb{N}$.
A subset $K \subset \mathbb{N} \times \mathbb{N} $ is said to have natural density $d(K)$ if 
\begin{center}
$d(K) = \lim\limits_{\stackrel{\stackrel{m\rightarrow
\infty}{n\rightarrow \infty}} ~} \frac{| K(m,n)|}{m.n} $,
\end{center}
where $K(m,n) = \{(j,k) \in \mathbb{N} \times \mathbb{N}: j \leq m, k \leq n;(j,k)\in K \} $ and $\left|K(m,n)\right|$ denotes number of elements of the set $K(m,n)$.  

A double sequence $x = \{x_{jk}\}_{ j, k \in \mathbb{N}}$ of real numbers is said to 
be statistically convergent to $\xi \in \mathbb{R}$, if
for any $\epsilon > 0$, we have $d(A(\epsilon)) = 0$, where
$A(\epsilon) = \{ (j,k) \in \mathbb{N}\times \mathbb{N} : \parallel x_{jk} - \xi \parallel \geq \varepsilon \}$.

More investigation and applications of statistical convergence of double sequences can be found in \cite{Da1,Tr2} and many others.

The notion of statistical convergence of double sequences has been further generalized to $I$-convergence of double sequences by Das et. al. \cite{Da2} using ideals in $\mathbb{N}\times\mathbb N$. For more details one can see \cite{Da1, Da3, Tr1} etc.

 Recently Belen et. al. \cite{Be} introduced the notion of ideal statistical convergence of double sequences, which is a new generalization of the notions of statistical convergence and usual convergence. More investigation and applications on this notion can be found in \cite{Be, Ya}.

The concept of rough convergence of double sequence was first introduced by Malik et. al.\cite{Ma1}. If $x=\left\{x_{jk}\right\}_{j,k\in\mathbb N}$ be a double sequence in some normed linear space $(X,\left\|.\right\|)$ and $r$ be a non negative real number, then $x$ is said to be $r$-convergent to $\xi\in X$ if for any $\varepsilon>0$, there exists $ m \in \mathbb{N} $ such that for all $ j, k \geq m $
\begin{center}
$| x_{jk} - \xi | < r+\epsilon $. 
\end{center} 

Further this notion of rough convergence of double sequence has been extended to rough statistical convergence of double sequence by Malik et. al.\cite{Ma2} using double natural density of the subsets of $\mathbb{N}\times\mathbb N$ in a similar way as the notion of convergence of double sequence in Pringsheim sense was generalized to statistical convergence of double sequences. Further the notion of rough statistical convergence of double sequences was generalized to rough $I$-convergence of double sequences by Dunder et. al. \cite{Du}. So it is quite natural to think , if the new notion of $I$-statistical convergence of double sequences can be introduced in the theory of rough convergence.

 In this paper we introduced and study the notion of rough $I$-statistical convergence of double sequences in a normed linear space $(X,\left\|.\right\|)$ which naturally extends both the notions of rough convergence as well as rough statistical convergence of double sequences in a new way. We also define the set of all rough $I$-statistical limits of a double sequence and investigate some topological properties of this set.

\section{\textbf{Basic Definitions and Notations}}
\begin{defn}\cite{Ma2}
Let $x = \{x_{jk}\}_{j,k \in \mathbb{N}}$ be a double sequence in a
normed linear space $(X, \parallel.\parallel)$ and $r$ be a non negative real number.
$x$ is said to be $r$- statistically convergent to $\xi$, denoted by $x \overset{r-st}\longrightarrow \xi$,
if for any $\varepsilon > 0$ we have $d(A(\varepsilon)) = 0$, where
$A(\varepsilon) = \{ (j,k) \in \mathbb{N}\times \mathbb{N} :~~ \parallel x_{jk} - \xi \parallel \geq r + \varepsilon \}$.
In this case $\xi$ is called the $r$-statistical limit of $x$.
\end{defn}
\begin{defn}
A class $I$ of subsets
of a nonempty set $X$ is said to be an ideal in $X$ provided

(i) $\phi\in I$.

(ii) $ A,B\in I$ $~$implies $A\bigcup B\in I$.

(iii) $ A\in I,B\subset A $$~$ implies $~~$   $B\in I$.

$I$ is called a nontrivial ideal if $X\notin I$.

\end{defn}
\begin{defn}
A non empty class $F$ of
subsets of a nonempty set $X$ is said to be a filter in $X$ provided

(i) $\phi\notin F$.

(ii) $A,B\in F$ $~$ implies $~~$ $A\bigcap B\in F$.

(iii) $A\in F,A\subset B$ $~$ implies $~~$ $B\in F$.

If $I$ is a nontrivial ideal in $X$, $X\neq\phi$, then the class
\begin{center}
$F(I)=\{ M \subset X  : M = X \setminus A$ for some $A \in I \}$
\end{center}
is a filter on $X$, called the filter associated with $I$.
\end{defn}
\begin{defn}
A nontrivial ideal $I$ in $X$ is called admissible if $\{x\} \in I$ for each $x \in X $.
\end{defn}
\begin{defn}
A nontrivial ideal $I$ on $\mathbb{N} \times \mathbb{N} $ is called strongly admissible if $\left\{i\right\}\times \mathbb{N}$ and $\mathbb{N}\times\left\{i\right\}$ belong to $I$ for each $i\in \mathbb{N}$.
\end{defn}
Clearly every strongly admissible ideal is admissible.
Throughout the paper we take $I$ as a strongly admissible ideal in $\mathbb{N} \times \mathbb{N} $.

\begin{defn}\cite{Ma3}
Let $ x = \{x_{jk}\}_{j,k \in \mathbb{N}} $ be a double sequence in a normed linear space $ (X, \parallel . \parallel) $ and $ r $ be a non negative real number. Then $ x $ is said to be rough $I$-convergent or $ r-I$-convergent to $ \xi $, 
denoted by $ x \overset{r-I}\longrightarrow \xi $, if for any $ \varepsilon > 0 $ we have $ \{(j,k) \in \mathbb{N} \times \mathbb{N} :\parallel x_{jk} - \xi \parallel \geq r + \varepsilon \} \in I $. In this case $ \xi $ is called rough $ I $-limit of $ x $ and $x$ is called rough $I$-convergent to $\xi$ with $r$ as roughness degree.
\end{defn}
Now we give the definition of $I$-asymptotic density of a subset of $\mathbb{N}\times\mathbb N$.
\begin{defn}
A subset $K \subset \mathbb{N} \times \mathbb{N} $ is said to have $I$-asymptotic density $d_I(K)$ if 
\begin{center}
$d_I(K) = I-\lim\limits_{\stackrel{\stackrel{m\rightarrow
\infty}{n\rightarrow \infty}} ~} \frac{| K(m,n)|}{m.n} $,
\end{center}
where $K(m,n) = \{(j,k) \in \mathbb{N} \times \mathbb{N}: j \leq m, k \leq n;(j,k)\in K \} $ and $\left|K(m,n)\right|$ denotes number of elements of the set $K(m,n)$. 
\end{defn}

\begin{defn}$[1]$
 A double sequence $x = \{x_{jk}\}_{j,k\in \mathbb{N}}$ of real numbers is  $I$-statistically convergent to $L$, and we write $x\stackrel{I-st}{\rightarrow}L$, provided that for any $\epsilon> 0$ and $\delta> 0$
\begin{center}
$\left\{\left(m,n\right)\in\mathbb{N}\times\mathbb{N} :\frac{1}{mn}\left|\left\{\left(j,k\right):\left|x_{jk}-L\right|\geq\epsilon,j\leq m,k\leq n\right\}\right|\geq\delta\right\}\in I$.
\end{center}
\end{defn}

\begin{defn}
Let $x=\{x_{jk}\}_{j,k \in \mathbb{N}}$ be a double sequence in  a normed
linear space $\left(X,\left\|.\right\|\right)$ and $r$ be a non
negative real number. Then $x$ is said to be rough
$I$-statistically convergent to $ \xi $ or
$r\mbox{-}I$-statistically convergent to $\xi$ if for any
$\varepsilon > 0$ and $ \delta > 0$
\begin{center}
$\{ (m,n) \in {\mathbb{N}\times\mathbb N}: \frac{1}{mn}|\{(j,k):j\leq m, k \leq n; \parallel x_{jk} - \xi \parallel \geq r+ \varepsilon\}|
\geq \delta \} \in I $.
\end{center}
\end{defn}
In this case $\xi$ is called the rough $I$-statistical limit of
$x=\{x_{jk}\}_{j,k \in \mathbb{N}}$ and we denote it by $ x
\overset{r\mbox{-}I \mbox{-}st}\longrightarrow \xi $.
\begin{rem}
Note that if $I$ is the ideal $I_0 = \{ A \subset \mathbb{N} \times \mathbb{N}: \exists 
m(A) \in \mathbb{N} ~such ~ that ~  i,j \geq m(A) \Rightarrow (i,j) \notin A \}$, then rough $I$-statistical 
convergence coincide with rough statistical convergence.
\end{rem}

Here $r$ in the above definition is called the roughness degree of
the rough $I$-statistical convergence. If $r=0$, we obtain the
notion of $I \mbox{-}$statistical convergence. But our main
interest is when $ r > 0 $. It may happen that a double sequence $x=\{x_{jk}\}_{j,k \in \mathbb{N}}$ is not $ I $-statistically convergent
in the usual sense, but there exists a double sequence $y=\{y_{jk}\}_{j,k \in \mathbb{N}}$, which is $ I $-statistically convergent and
satisfying the condition $\| x_{jk} -y_{jk} \|\leq r $ for all  $j,k$. Then $ x $ is
rough $I$-statistically convergent to the same limit.

From the above definition it is clear that the rough
$I$-statistical limit of a double sequence is not unique. So we consider
the set of rough $I$-statistical limits of a double sequence $x$ and we
use the notation $I\mbox{-}st\mbox{-}\mbox{LIM}_x^r$ to denote the
set of all rough $I$-statistical limits of a double sequence $x$. We say
that a double sequence $x$ is rough $I$-statistically convergent if
$I\mbox{-}st\mbox{-}\mbox{LIM}_x^r\neq \phi$.

Throughout the paper $X$ denotes a normed linear space $(X, \|. \|)$ and $x$ denotes the double sequence $\{x_{jk}\}_{j,k \in \mathbb{N}}$ in $X$ .

 We now provide an example to show that there exists a double sequence which is neither
 rough statistically convergent nor  $I$-statistically convergent but is rough $I$-statistically convergent.

\begin{Example}
Let $ I $ be a nontrivial strongly admissible ideal in $\mathbb{N}\times\mathbb N$ which contains at least one infinite subset of $\mathbb{N}\times \mathbb N$. Choose an
infinite set $ A \subset {\mathbb{N}\times\mathbb N} $, whose $I$-asymptotic density is zero but
double natural density does not exists. We define a double sequence $x=\{x_{jk}\}_{j,k \in \mathbb{N}}$ in the following way

\[ x_{jk} = \left\{
  \begin{array}{l l}
    (-1)^{j+k}, & \quad \text{if $ (j,k)\notin A$ }\\
    jk, & \quad \text{if $(j,k)\in A$}.
  \end{array} \right.\]\\

Then $x$ neither rough statistically convergent nor
$I$-statistically convergent but
\[ I\mbox{-}st\mbox{-}LIM^r_x = \left\{
  \begin{array}{l l}
    \emptyset, & \quad \text{if $ r < 1 $ }\\
    \left[1-r, r-1\right], & \quad \text{otherwise.}
  \end{array} \right.\]\\

\end{Example}


\section{\textbf{Main Results}}

In this section we discuss some basic properties of rough $I$-statistical convergence of double sequences.


\begin{thm}
Let $x = \{x_{jk} \}_{j, k \in \mathbb{N}} $ be a double sequence in $X$ and
$r \geq 0 $. Then diam ($I\mbox{-}st\mbox{-}\mbox{LIM}_x^r) \leq 2r $.
In particular if $x$ is $I$-statistically convergent to $\xi$,
then $I\mbox{-}st\mbox{-}\mbox{LIM}_x^r
={\overline{{B_r}(\xi)}}(=\left\{y\in
X:\left\|y-\xi\right\|\leq r\right\})$ and so diam
($I\mbox{-}st\mbox{-}\mbox{LIM}_x^r) = 2r $.
\end{thm}

\begin{proof}
If possible let, diam ($I\mbox{-}st\mbox{-}\mbox{LIM}_x^r) > 2r $. Then there exist $ y,z \in I\mbox{-}st\mbox{-}\mbox{LIM}_x^r $ such that $\| y - z \| > 2r $. Now choose $ \varepsilon > 0$
so that $ \varepsilon < \frac{\| y - z \|}{2} - r $.
Let
$ A = \{(j,k) \in \mathbb{N\times N} : \|x_{jk} - y \| \geq r +\varepsilon  \}$
and
$ B = \{(j,k) \in \mathbb{N\times N} : \|x_{jk} - z \| \geq r + \varepsilon \}$.
Then
\begin{center}
$ \frac{1}{mn} |\{ (j,k):j\leq m;k \leq n,(j,k) \in A \cup B \}| 
\leq \frac{1}{mn}|\{(j,k):j\leq m; k \leq n,(j,k) \in A\}| + \frac{1}{mn}|\{(j,k):j\leq m; k \leq n ,(j,k)\in B \}|$,
\end{center}
and so by the property of $I$-convergence \\
$ I \mbox{-}\lim\limits_{\stackrel{\stackrel{m\rightarrow
\infty}{n\rightarrow \infty}} ~} \frac{1}{mn} |\{ (j,k):j\leq m;k \leq n,(j,k) \in A \cup B \}| \leq 
I \mbox{-}\lim\limits_{\stackrel{\stackrel{m\rightarrow
\infty}{n\rightarrow \infty}} ~}\frac{1}{mn}|\{(j,k):j\leq m; k \leq n,(j,k) \in A\}| + \frac{1}{mn}|\{(j,k):j\leq m; k \leq n ,(j,k)\in B \}|= 0 $.
Thus $\{(m,n )\in \mathbb{N\times N} : \frac{1}{mn} |\{(j,k):j\leq m; k \leq n, (j,k)\in A \cup B \}| \geq \delta\} \in I $ for all $\delta>0$.
Let $ K = \{(m, n) \in \mathbb{N\times N} : \frac{1}{mn} |\{(j,k):j\leq m;k \leq n,(j,k )\in A \cup B \}| \geq \frac{1}{2}\} $. Clearly $K\in I$,
so choose $ (m_0, n_0) \in \mathbb{N\times N} \setminus K $. Then $ \frac{1}{m_0n_0}|\{(j,k):j\leq m_0;k \leq n_0,(j, k) \in A \cup B \}| < \frac{1}{2}$. 
So $\frac{1}{m_0 n_0} | \{(j,k):j\leq m_0;k \leq n_0,(j, k) \notin A \cup B \}| \geq 1- \frac{1}{2} = \frac{1}{2} $
i.e., $ \{(j, k) :(j, k) \notin A \cup B \}$ is a nonempty set. Take $(j_0, k_0 )\in \mathbb{N\times N}$ such that $(j_0, k_0) \notin A \cup B$. Then $(j_0, k_0) \in A^c \cap B^c$ and hence $ \|x_{j_0 k_0} - y \| < r + \varepsilon $ and $ \| x_{j_0 k_0} - z \| < r + \varepsilon $. So
$\|y - z \| \leq \|x_{j_0 k_0} -y \| + \|x_{j_0k_0} - z \| \leq 2(r + \varepsilon) < \| y - z \|$,
which is absurd. Therefore diam ($I\mbox{-}st\mbox{-}\mbox{LIM}_x^r) \leq 2r $.\\

If $ I\mbox{-}st\mbox{-}\lim x = \xi $, then we proceed as follows.
Let $ \varepsilon > 0 $ and $ \delta > 0 $ be given. Then
$ A = \{(m,n) \in \mathbb{N\times N}: \frac{1}{mn}|\{ (j,k):j\leq m;k \leq n , \|x_{jk} - \xi \|\geq \varepsilon \}|\geq \delta \} \in I $.
Then for $ (m,n) \notin A $ we have 
$\frac{1}{mn}|\{(j,k):j\leq m;k \leq n, \|x_{jk} - \xi \| \geq \varepsilon \} < \delta $,
i.e., 
\begin{eqnarray}
\frac{1}{mn} |\{(j,k):j\leq m;k \leq n, \|x_{jk} - \xi \|  < \varepsilon \} \geq 1 - \delta. 
\end{eqnarray}
Now for each $y\in {\overline{{B_r}(\xi)}}(=\left\{y\in
X:\left\|y-\xi\right\|\leq r\right\})$ we have 
\begin{eqnarray}
\|x_{jk} - y \| \leq \| x_{jk} - \xi \| + \|\xi - y \| \leq \|x_{jk} - \xi \| + r.
\end{eqnarray}
Let $B_{mn }= \left\{ (j,k):j\leq m;k \leq n ,\|x_{jk} - \xi \| < \varepsilon \right\}$. Then for $ (j,k) \in B_{mn}$ we have
$\|x_{jk} -y \| < r + \varepsilon$.
Hence $ B_{mn} \subset \{ (j,k):j\leq m;k \leq n, \|x_{jk} - y \|< r + \varepsilon\}$. This implies,
$ \frac{|B_{mn}|}{mn} \leq \frac{1}{mn}|\{(j,k):j\leq m;k \leq n , \|x_{jk} -y \| < r + \varepsilon\}|$
i.e., $\frac{1}{mn}|\{(j,k):j\leq m;k \leq n,\|x_{jk} - y\|< r + \varepsilon \}| \geq 1 -\delta $.
Thus for all $ (m,n) \notin A $, $ \frac{1}{mn} |(j,k):j\leq m;k \leq n, \|x_{jk} - y \|\geq r + \varepsilon \}|<1-(1 - \delta) $.
Hence we have 
$\{(m,n)\in\mathbb{N\times N}: \frac{1}{mn} |\{(j,k):j\leq m;k \leq n,\|x_{jk} - y \| \geq r +\varepsilon \}| \geq \delta\} \subset A $.
Since $ A \in I $, so 
$\{(m,n)\in\mathbb{N\times N}: \frac{1}{mn} |\{(j,k):j\leq m;k \leq n,\|x_{jk} - y \| \geq r + \varepsilon \}| \geq \delta\} \in I $.
This shows that $ y \in I\mbox{-}st\mbox{-}LIM_x^r$. Therefore $I\mbox{-}st\mbox{-}\mbox{LIM}_x^r
\supset{\overline{{B_r}(\xi)}} $.

 Conversely, let $ y \in I\mbox{-}st\mbox{-}LIM_x^r$. If possible , let $\left\|y-\xi\right\|>r$. Take $\varepsilon=\frac{\left\|y-\xi\right\|-r}{2}$. Let
$ K_1 = \{(j,k) \in \mathbb{N\times N} : \|x_{jk} - y \| \geq r +\varepsilon  \}$
and
$ K_2= \{(j,k) \in \mathbb{N\times N} : \|x_{jk} - \xi \| \geq \varepsilon \}$.
Then
\begin{center}
$ \frac{1}{mn} |\{ (j,k):j\leq m;k \leq n,(j,k) \in K_1 \cup K_2\}| 
\leq \frac{1}{mn}|\{(j,k):j\leq m; k \leq n,(j,k) \in K_1\}| + \frac{1}{mn}|\{(j,k):j\leq m; k \leq n ,(j,k)\in K_2 \}|$,
\end{center}
and so by the property of $I$-convergence \\
$ I \mbox{-}\lim\limits_{\stackrel{\stackrel{m\rightarrow
\infty}{n\rightarrow \infty}} ~} \frac{1}{mn} |\{ (j,k):j\leq m;k \leq n,(j,k) \in K_1 \cup K_2 \}| \leq 
I \mbox{-}\lim\limits_{\stackrel{\stackrel{m\rightarrow
\infty}{n\rightarrow \infty}} ~}\frac{1}{mn}|\{(j,k):j\leq m; k \leq n,(j,k) \in K_1\}| + \frac{1}{mn}|\{(j,k):j\leq m; k \leq n ,(j,k)\in K_2 \}|= 0 $. Let $ K = \{(m, n) \in \mathbb{N\times N} : \frac{1}{mn} |\{(j,k):j\leq m;k \leq n,(j,k )\in K_1 \cup K_2 \}| \geq \frac{1}{2}\} $. Clearly $K\in I$ and we choose $ (m_0, n_0) \in \mathbb{N\times N} \setminus K $. Then $ \frac{1}{m_0n_0}|\{(j,k):j\leq m_0;k \leq n_0,(j, k) \in K_1 \cup K_2\}| < \frac{1}{2}$. 
So $\frac{1}{m_0 n_0} | \{(j,k):j\leq m_0;k \leq n_0,(j, k) \notin K_1 \cup K_2 \}| \geq 1- \frac{1}{2} = \frac{1}{2} $
i.e., $ \{(j, k) :(j, k) \notin K_1 \cup K_2 \}$ is a nonempty set. We choose $(j_0, k_0 )\in \mathbb{N\times N}$ such that $(j_0, k_0) \notin K_1 \cup K_2$. Then $(j_0, k_0) \in {K_1}^c \cap {K_2}^c$ and hence $ \|x_{j_0 k_0} - y \| < r + \varepsilon $ and $ \| x_{j_0 k_0} - \xi \| <  \varepsilon $. So
\begin{center}
$\|y - \xi \| \leq \|x_{j_0 k_0} -y \| + \|x_{j_0k_0} - \xi \| \leq r +2 \varepsilon < \| y - z \|$,
\end{center}
which is absurd. Therefore $\left\|y-\xi\right\|\leq r$  and so $y\in {\overline{{B_r}(\xi)}}$. Consequently we have $I\mbox{-}st\mbox{-}\mbox{LIM}_x^r
={\overline{{B_r}(\xi)}}$ and this completes the proof.
\end{proof}

\begin{defn}
A double sequence $ x=\left\{x_{jk}\right\}_{j,k\in\mathbb N}$ is said to be $ I $- statistically bounded if there exists a positive number $ T $ such that for any $ \delta > 0 $ the set $ A = \{(m,n)\in\mathbb{N\times N}: \frac{1}{mn}|\{ (j,k):j\leq m;k \leq n,\|x_{jk}\| \geq T \}|  \geq \delta\}\in I $.
\end{defn}
The next result provides a relationship between boundedness and rough $I$-statistical convergence of double sequences.
\begin{thm}
If a double sequence $x = \{ x_{jk}\}_{ j,k \in \mathbb{N}}$ is
bounded then there exists $r \geq 0$ such that $I\mbox{-}st\mbox{-}LIM_x^r \neq
\emptyset$.
\end{thm}
\begin{proof}
Let $x = \{ x_{jk}\}_{ j,k \in \mathbb{N}}$ be a bounded double sequence. Then there exists a positive
real number $T$ such that $\parallel x_{jk} \parallel < T$ for all
$(j,k) \in \mathbb{N} \times \mathbb{N}$. Let $\varepsilon>0$ be given. Then 
\begin{center}
$\left\{(j,k):\left\|x_{jk}-0\right\|\geq T+\varepsilon\right\}=\emptyset$.
\end{center}
 Therefore $0 \in I\mbox{-}st\mbox{-}LIM_x^T$ and so $I\mbox{-}st\mbox{-}LIM_x^T
\neq \emptyset$.
\end{proof}
\begin{note}
The converse of the above theorem is not true. For example,
let us consider the double sequence $x = \{ x_{jk}\}_{ j,k \in
\mathbb{N}}$ in $\mathbb{R}$ defined by
\begin{eqnarray*}
x_{jk} &=& jk ~~~~, \mbox{if}~ j ~and~ k~ are~ squares \\
       &=& 3 ~~~~, ~ otherwise.
\end{eqnarray*}
Then $I\mbox{-}st\mbox{-}LIM_x^0= \{3\} \neq \emptyset$ but the double sequence
$x$ is unbounded.
\end{note}

We now show that the converse of Theorem 3.2 is true if the double sequence is $I$-statistically bounded.


\begin{thm}
A  double sequence $ x=\left\{x_{jk}\right\}_{j,k\in\mathbb N} $ in $X$ is
$I$-statistically bounded if and only if there exists a non
negative real number $ r $ such that $
I\mbox{-}st\mbox{-}LIM_x^r \neq \emptyset $.
\end{thm}

\begin{proof}
Let $ x=\left\{x_{jk}\right\}_{k\in\mathbb N} $ be an
$I$-statistically bounded double sequence in $X$. Then there exists a
positive real number $T$ such that for $\delta > 0 $ we have $\{ (m,n)
: \frac{1}{mn}|\{(j,k):j\leq m; k \leq n, \|x _{jk} \| \geq T \}| \geq \delta \} \in
I $. Let $ A = \{ (j,k): \|x_{jk} \| \geq T \} $. Then $
I\mbox{-}\underset{ m,n \rightarrow \infty}{\lim}\frac{1}{mn}|\{(j,k) :j\leq m;k
\leq n,(j, k) \in A \}| = 0 $. Let $ r' = sup\{\| x_{jk}\| : (j,k) \in
A^c\}$. Then $ 0\in I\mbox{-}st\mbox{-}LIM^{r'}_x $. Hence $ I\mbox{-}st\mbox{-}LIM^r_x \neq \emptyset
$ for $ r = r'$.

Conversely, let $ I\mbox{-}st\mbox{-}LIM_x^r \neq \emptyset $ for
some $ r \geq 0 $. Let $ \xi \in I\mbox{-}st\mbox{-}LIM_x^r $. Choose
$\varepsilon = \| \xi \| $. Then for each $ \delta
> 0 $, $ \{ (m,n)\in \mathbb{N\times N}: \frac{1}{mn}|\{(j,k):j\leq m;k \leq n,
\|x_{jk} - \xi \| \geq r + \varepsilon \}| \geq \delta \} \in I $.
Now taking $ T = r + 2 \| \xi \|$, we have $ \{ (m,n) \in \mathbb{N\times N}:
\frac{1}{mn} | \{ (j,k):j\leq m;k \leq n , \|x_{jk}\| \geq T \}| \geq \delta \} \in
I $. Therefore $ x $ is $I$-statistically bounded.
\end{proof}

Now let $ \{j_p\}_{p \in \mathbb{N}} $ and $ \{k_q\}_{q \in \mathbb{N}}$ be two strictly increasing sequences
of natural numbers. If $ x = \{ x_{jk}\}_{j,k \in \mathbb{N}}$ is a double sequence in $(X, \parallel.\parallel)$,
then we define $\{x_{j_{p}k_{q}}\}_{p,q \in \mathbb{N}}$ as a subsequence of $x$.
\begin{defn}
A subsequence $x^{\prime} = \{x_{j_{p}k_{q}}\}_{p,q \in \mathbb{N}}$ of a double sequence $x = \{x_{jk}\}_{j,k \in \mathbb{N}}$
is called a $I$-dense subsequence if $d_I(\{(j_{p}, k_{q}); p, q \in \mathbb{N} \}) = 1 $. 
\end{defn}

In (\cite{Ma1}, Theorem 3.4), Malik et. al. have already shown that if $x^{'} = \{x_{{j_p}{k_q}}\}_{p,q \in \mathbb{N}}$ is a subsequence
of $ x = \{ x_{jk}\}_{j,k \in \mathbb{N}}$, then $LIM_x^r \subset LIM_{x^{'}}^r$. But this result is not true for rough $I$-statistical convergence. 
To show this we consider the following example.

\begin{Example}
Consider the double sequence $ x = \{ x_{jk}\}_{j,k \in \mathbb{N}}$ in $\mathbb{R}$ defined by
\begin{eqnarray*}
x_{jk} &=& jk ~~~~, \mbox{if j and k both are cubes} \\
       &=& 0 ~~~~, \mbox{otherwise}.
\end{eqnarray*}
Then $x^{\prime} = \{x_{j_{p}k_{q}}\}_{p,q \in \mathbb{N}}$ is a subsequence of $x$, where
$j_{p} = p^{3}, p \in \mathbb{N} $ and $k_{q} = q^{3}, q \in \mathbb{N} $ and for all $ r \geq 0 $,
$I\mbox{-}st\mbox{-}LIM_x^r = [-r,r]$ but $I\mbox{-}st\mbox{-}LIM_{x'}^r = \emptyset $.
\end{Example}
We now present $I$-statistical analogue of Theorem 3.4 \cite{Ma1} in the following form.

\begin{thm}
If $x'=\{x_{{j_p}{k_q}}\}_{p,q\in \mathbb{N}}$ is a $I$-dense subsequence of
$x = \{ x_{jk}\}_{ j,k \in \mathbb{N}}$, then $ I\mbox{-}st\mbox{-}LIM_x^r \subset I\mbox{-}st\mbox{-}LIM_{x'}^r $.
\end{thm}

\begin{proof}
Let $\xi \in I\mbox{-}st\mbox{-}LIM_x^r$. Then for any $\varepsilon > 0$, $d_I(A(\varepsilon)) = 0$, where
$A(\varepsilon) = \{(j,k) \in \mathbb{N} \times \mathbb{N}; \parallel x_{jk} - \xi \parallel \geq r + \varepsilon\}$.
Then $d_I(A^c(\varepsilon)) = 1$.\\
Since $x'=\{x_{{j_p}{k_q}}\}_{p,q\in \mathbb{N}} $ is a $I$-dense subsequence of $x$, so $d_I(K) = 1 $, where $K = \{(j_{p}, k_{q} ); p, q \in \mathbb{N} \}$.
Then $d_I(A^{c} (\varepsilon) \cap K ) = 1 $. \\
Let $A^{'}(\varepsilon) = \{(p,q) \in \mathbb{N} \times \mathbb{N}; \parallel x_{{j_p}{k_q}} - \xi \parallel \geq r + \varepsilon\}$. Now $\{(p,q); \parallel x_{{j_p}{k_q}} - \xi \parallel < r + \varepsilon \} \supset A^{c} (\varepsilon) \cap K $. Therefor $d_I(A^{\prime}(\varepsilon))^{c} = 1 $ and so $d_I(A^{\prime}(\varepsilon)) =0 $-which 
implies $ \xi \in I\mbox{-}st\mbox{-}LIM_{x'}^r $.
\end{proof}


\begin{thm}
Let $ x=\left\{x_{jk}\right\}_{j,k\in\mathbb N}$ be double sequence and $ r \geq 0 $ be a real number. Then the rough $I$-statistical limit set of the double sequence $ x $ i.e., the set
$ I\mbox{-}st\mbox{-}LIM_x^r $ is closed.
\end{thm}

\begin{proof}
If $ I\mbox{-}st\mbox{-}LIM_x^r  = \emptyset $, then nothing to prove.\\
Let us assume that $ I\mbox{-}st\mbox{-}LIM_x^r \neq \emptyset $. Now consider a double sequence $\{y_{jk}\}_{j,k \in \mathbb{N}}$
in $ I\mbox{-}st\mbox{-}LIM_x^r $ with $ \lim\limits_{\stackrel{\stackrel{j\rightarrow
\infty}{k\rightarrow \infty}} ~}y_{jk} = y $.
Choose $ \varepsilon > 0 $ and $ \delta > 0 $. Then there exists $ i_{\frac{\varepsilon}{2}} \in \mathbb{N} $ such that $ \|y_{jk} - y \| < \frac{\varepsilon}{2} $ for all $ j > i_{\frac{\varepsilon}{2}}$ and  $ k > i_{\frac{\varepsilon}{2}}$. 
Let $ j_0 > i_{\frac{\varepsilon}{2}}$ and  $ k_0 > i_{\frac{\varepsilon}{2}}$. Then $ y_{j_0k_0} \in I\mbox{-}st\mbox{-}LIM_x^r $.
Consequently, we have 
$ A = \{ (m,n)\in\mathbb{N\times N} : \frac {1}{mn} | \{ (j,k):j\leq m;k \leq n , \|x_{jk} - y_{j_0k_0}\| \geq r + \frac{\varepsilon}{2}\}| \geq \delta \} \in I $. Clearly $ M = \mathbb{N\times N} \setminus A $ is nonempty, choose $ (m,n) \in M $.
We have 
\begin{eqnarray*}
\frac{1}{mn} | \{ (j,k):j\leq m;k \leq n , \|x_{jk} - y_{j_0k_0} \| \geq r + \frac{\varepsilon}{2} \}|  < \delta.
\end{eqnarray*}
\begin{eqnarray*}
\Rightarrow\frac{1}{mn} | \{ (j,k):j\leq m;k \leq n , \|x_{jk} - y_{j_0k_0} \| < r + \frac{\varepsilon}{2} \}|  \geq 1 - \delta.
\end{eqnarray*}
Put $ B_{mn} = \{ (j,k):j\leq m;k \leq n , \|x_{jk} - y_{j_0k_0} \| < r + \frac{\varepsilon}{2} \} $. Choose $(j, k) \in B_{mn} $.
Then 
\begin{eqnarray*}
\|x_{jk} - y \| \leq \| x_{jk} - y_{j_0k_0}\| + \|y_{j_0k_0} - y \| < r + \frac{\varepsilon}{2} + \frac{\varepsilon}{2} 
= r + \varepsilon.
\end{eqnarray*}
Hence $ B_{mn} \subset \{ (j,k):j\leq m;k \leq n  , \|x_{jk} - y \| < r + \varepsilon \}$, which implies
\begin{center}
$ 1-\delta \leq\frac{ |B_{mn}|}{mn} \leq \frac{1}{mn}\left|\{ (j,k):j\leq m;k \leq n , \|x_{jk} - y \| < r + \varepsilon\}\right| $.
\end{center}
Therefore $\frac{1}{mn}|\{ (j,k):j\leq m;k \leq n , \|x_{jk} - y \| \geq r + \varepsilon\}| < 1-(1-\delta)=\delta$.\\
Thus we have
\begin{eqnarray*}
\{(m, n): \frac{1}{mn} | \{ (j,k):j\leq m;k \leq n, \|x_{jk} - y \| \geq r + \varepsilon \}| \geq \delta \} \subset A \in I.
\end{eqnarray*}
This shows that $ y \in I\mbox{-}st\mbox{-}LIM_x^r $. Hence $ I\mbox{-}st\mbox{-}LIM_x^r $ is a closed set.

\end{proof}


\begin{thm}
Let $ x=\left\{x_{jk}\right\}_{j,k\in\mathbb N}$ be double sequence and $ r \geq 0 $ be a real number. Then the rough $I$-statistical limit set 
$ I\mbox{-}st\mbox{-}LIM_x^r $ of the double sequence $ x $ is a convex set.

\end{thm}

\begin{proof}
Let $ y_0, y_1 \in I\mbox{-}st\mbox{-}LIM_x^r $. Let $ \varepsilon > 0 $ be given. Let\\
\begin{eqnarray*}
A_0 = \{ (j,k) \in \mathbb{N\times N} : \| x_{jk} - y_0 \| \geq r + \varepsilon \}~ \\
A_1 = \{ (j,k) \in \mathbb{N\times N} : \| x_{jk} - y_1 \| \geq r + \varepsilon \}. 
\end{eqnarray*}
Then by theorem 3.1 for $ \delta > 0 $ we have 
$ \{ (m,n) : \frac{1}{mn} |\{(j,k):j\leq m;k \leq n ,(j, k) \in A_0 \cup A_1 \}| \geq \delta \} \in I $.
Choose $ 0 < {\delta}_1 < 1 $ such that $ 0 < 1 - {\delta}_1 < \delta $.
Let $ A = \{ (m,n) : \frac{1}{mn} |\{(j,k):j\leq m;k \leq n ,(j, k) \in A_0 \cup A_1 \}| \geq {1-\delta}_1 \}$.
Then $ A \in I $. Now for all $ (m,n) \notin A $ we have 
\begin{eqnarray*}
\frac{1}{mn} |\{(j,k):j\leq m;k \leq n ,(j, k )\in A_0 \cup A_1 \}| < 1 - {\delta}_1 
\end{eqnarray*}
\begin{eqnarray*}
\Rightarrow\frac{1}{mn} |\{(j,k):j\leq m;k \leq n ,(j, k) \notin A_0 \cup A_1 \}| \geq \left\{1-(1 - {\delta}_1)\right\}=\delta_1. 
\end{eqnarray*}
Therefore $ \{ (j,k):(j, k) \notin A_0 \cup A_1 \}$ is a nonempty set. Let us take $ (j_0,k_0) \in {A_0}^c \cap {A_1}^c $ and
$ 0 \leq \lambda \leq 1 $. Then
\begin{eqnarray*}
 \|x_{j_0k_0} - [(1 - \lambda)y_0 + {\lambda}y_1] \|
& = & \|(1 - \lambda)x_{j_0k_0} + {\lambda}x_{j_0k_0} - [(1 - \lambda)y_0 + {\lambda}y_1] \|\\ 
& \leq & (1 - \lambda)\|x_{j_0k_0} - y_0 \| + \lambda \|x_{j_0k_0} - y_1\|\\
& < & (1 - \lambda)(r + \varepsilon) + \lambda(r + \varepsilon) = r + \varepsilon.
\end{eqnarray*}
 Let $ B = \{ (j,k) \in \mathbb{N\times N}: \| x_{jk} - [(1-\lambda)y_0 + {\lambda}y_1]\|\geq r + \varepsilon\}$.
Then clearly, $ {A_0}^c \cap {A_1}^c \subset B^c $. So for $ (m,n) \notin A $,
\begin{eqnarray*}
{\delta}_1 \leq \frac{1}{mn}|\{ (j,k):j\leq m;k \leq n,(j, k) \notin A_0 \cup A_1 \} \leq \frac{1}{mn}|\{ (j,k):j\leq m;k \leq n ,(j, k) \notin B \}.
\end{eqnarray*}
\begin{eqnarray*}
\Rightarrow\frac{1}{mn}|\{ (j,k):j\leq m;k \leq n,(j, k) \in B\}| < 1 - {\delta}_1 < \delta.
\end{eqnarray*}
 Thus $ A^c \subset \{ (m,n) : \frac{1}{mn}|\{ (j,k):j\leq m;k \leq n,(j, k )\in B\}| < \delta \}$.
Since $ A^c \in F(I) $, so $\{(m, n) : \frac{1}{mn}|\{ (j,k):j\leq m;k \leq n,(j, k) \in B\}| < \delta \} \in F(I)$ and so $\{ (m,n) : \frac{1}{mn}|\{ (j,k):j\leq m;k \leq n,(j, k) \in B\}| \geq \delta \} \in I $.
This completes the proof.

\end{proof}


\begin{thm}
Let $ r > 0 $. Then a double sequence $ x=\left\{x_{jk}\right\}_{j,k\in\mathbb N}$ is rough $I$-statistically convergent to $ \xi $ if and only if there exists a double sequence $ y = \{y_{jk}\}_{j,k \in \mathbb{N}} $ such that $ {I\mbox{-}st\mbox{-}\lim}~ y = \xi $ and $ \parallel x_{jk} - y_{jk} \parallel \leq r $ for all $(j, k) \in \mathbb{N\times N} $.
\end{thm}

\begin{proof}
Let $ y = \{y_{jk}\}_{j,k \in \mathbb{N}}$ be a double sequence in $X$, which is $I$-statistically convergent to $\xi$ and $\|x_{jk} - y_{jk} \| \leq r $ for all $(j, k) \in \mathbb{N\times N} $. Then for any $ \varepsilon > 0 $ and $ \delta > 0 $ the set
$ A = \{ (m,n)\in\mathbb{N\times N}: \frac{1}{mn}|\{ (j,k):j\leq m;k \leq n , \|y_{jk} - \xi \|\geq \varepsilon \} \geq \delta \} \in I $
. Let $ (m,n) \notin A $. Then we have
\begin{eqnarray*}
\frac{1}{mn} |\{(j,k):j\leq m;k \leq n , \|y_{jk} - \xi \| \geq \varepsilon \}| < \delta
\end{eqnarray*}
\begin{eqnarray*}
\Rightarrow\frac{1}{mn}|\{ (j,k):j\leq m;k \leq n , \|y_{jk} - \xi \| < \varepsilon \}|  \geq 1 - \delta.
\end{eqnarray*}
Let $ B_{mn} = \{ (j,k):j\leq m;k \leq n  , \|y_{jk} - \xi \| < \varepsilon \}$. Then for $ (j,k) \in B_{mn} $, we have
\begin{eqnarray*}
\|x_{jk} - \xi \| \leq \| x_{jk} - y_{jk} \| + \|y_{jk} - \xi \| < r + \varepsilon.
\end{eqnarray*}
Therefore, 
\begin{eqnarray*}
B_{mn} \subset \{ (j,k):j\leq m;k \leq n, \|x_{jk} - \xi \| < r + \varepsilon \}.
\end{eqnarray*} 
\begin{eqnarray*}
\Rightarrow\frac{|B_{mn}|}{mn} \leq \frac{1}{mn} | \{ (j,k):j\leq m;k \leq n , \|x_{jk} - \xi \| < r + \varepsilon \}|.
\end{eqnarray*}
\begin{eqnarray*}
\Rightarrow\frac{1}{mn} |\{ (j,k):j\leq m;k \leq n , \|x_{jk} - \xi \| < r + \varepsilon \}| \geq 1 - \delta.
\end{eqnarray*}
\begin{eqnarray*}
\Rightarrow\frac{1}{mn} |\{ (j,k):j\leq m;k \leq n , \| x_{jk} - \xi \| \geq r + \varepsilon \}| < 1 - (1 - \delta) = \delta.
\end{eqnarray*}
Thus, $\{ (m,n)\in\mathbb{N\times N}: \frac{1}{mn}|\{ (j,k):j\leq m;k \leq n , \|x_{jk} - \xi \| \geq r + \varepsilon \}|  \geq \delta \} \subset A $ and since $ A \in I $, we have $ \{ (m,n)\in\mathbb{N\times N} : \frac{1}{mn}|\{ (j,k):j\leq m;k \leq n , \|x_{jk} - \xi \| \geq r + \varepsilon \}|\geq\delta\} \in I $. Hence $ {I\mbox{-}st\mbox{-}\lim}~ x = \xi $.\\
Conversely, suppose that $ {I\mbox{-}st\mbox{-}\lim}~ x = \xi $. Then for $ \varepsilon > 0 ~~\mbox{and}~~ \delta >0 $, 
\begin{eqnarray*}
A = \{ (m,n)\in\mathbb{N\times N} : \frac{1}{mn} | \{ (j,k):j\leq m;k \leq n  , \|x_{jk} - \xi \| \geq r + \varepsilon \}| \geq \delta \} \in I.
\end{eqnarray*}
Let $ (m,n )\notin A $. Then
\begin{eqnarray*}
\frac{1}{mn} | \{ (j,k):j\leq m;k \leq n  , \|x_{jk} - \xi \| \geq r + \varepsilon \}| < \delta.
\end{eqnarray*}
\begin{eqnarray*}
\Rightarrow\frac{1}{mn} | \{ (j,k):j\leq m;k \leq n  , \|x_{jk} - \xi \| < r + \varepsilon \}| \geq 1 - \delta .
\end{eqnarray*}
Let $ B_{mn}= \{ (j,k):j\leq m;k \leq n  , \|x_{jk} - \xi \| < r + \varepsilon \}$. Now we define a double sequence $y=\{y_{jk}\}_{j,k\in\mathbb N}$ as follows,
\[ y_{jk} = \left\{ 
  \begin{array}{l l}
    \xi, & \quad \text{if $\|x_{jk} - \xi \| \leq r $ }\\
    x_{jk} + r \frac{\xi - x_{jk}}{\|x_{jk} - \xi \|}, & \quad \text{otherwise.}
  \end{array} \right.\]\\
  Then 
	\begin{eqnarray*}
 \|y_{jk} -\xi \| & = &  \left\{
  \begin{array}{l l}
    0, & \quad \text{if $\|x_{jk} - \xi \| \leq r $ }\\
    \|x_{jk} - \xi  + r \frac{\xi - x_{jk}}{\|x_{jk} - \xi \|}\|, & \quad \text{otherwise.}
  \end{array} \right.\\
  \\
& = &  \left\{
  \begin{array}{l l}
    0, & \quad \text{if $\|x_{jk} - \xi \| \leq r $ }\\
    \| x_{jk} - \xi \| - r , & \quad \text{otherwise.}
  \end{array} \right.\\
\end{eqnarray*}
Let $ (j,k) \in B_{mn} $. 
Then 
\begin{eqnarray*}
\|y_{jk} - \xi \| = 0 , ~~\mbox{if}~~\|x_{jk} -\xi \| \leq r 
\end{eqnarray*}
and
\begin{eqnarray*}
\|y_{jk} - \xi \| < \varepsilon,   ~~\mbox{if}~~ r < \|x_{jk} - \xi \| < r + \varepsilon.
\end{eqnarray*}
Then $ B_{mn} \subset \{ (j,k):j\leq m;k \leq n  , \|y_{jk} - \xi \| < \varepsilon \} $. This implies 
\begin{center}
$ \frac{|B_{mn}|}{mn} \leq \frac{1}{mn} |\{ (j,k):j\leq m;k \leq n , \|y_{jk} - \xi\| < \varepsilon\} |$.
\end{center}
Hence, 
\begin{center}
$ \frac{1}{mn}|\{(j,k):j\leq m;k \leq n , \|y_{jk} - \xi \| < \varepsilon \}| \geq 1 - \delta $
\end{center}
\begin{center}
$ \Rightarrow \frac{1}{mn}|\{(j,k):j\leq m;k \leq n , \|y_{jk} - \xi \| \geq \varepsilon \}| < 1 -(1 - \delta) = \delta $. 
\end{center}
Thus $\{(m,n)\in\mathbb{N\times N}: \frac{1}{mn}|\{(j,k):j\leq m;k \leq n  , \|y_{jk} - \xi \| \geq \varepsilon\} | \geq \delta \} \subset A $.
Since $ A \in I $, we have 
\begin{center}
$ \{ (m,n)\in\mathbb{N\times N}: \frac{1}{mn}|\{ (j,k):j\leq m;k \leq n , \|y_{jk} - \xi \| \geq \varepsilon \}| \geq \delta \} \in I $.
\end{center}
So $ {I\mbox{-}st\mbox{-}\lim}~ y = \xi $.
\end{proof}

\begin{defn}
A point $\lambda \in X $ is said to be an $I$-statistical cluster point of a double sequence $ x=\left\{x_{jk}\right\}_{j,k\in\mathbb N}$ in $X$ if for any $\varepsilon > 0 $
\begin{eqnarray*}
d_I(\{(j,k):\|x_{jk} - \lambda \| < \varepsilon \}) \neq 0
\end{eqnarray*}
where 
$ d_I(A) = I - \lim\limits_{\stackrel{\stackrel{m\rightarrow
\infty}{n\rightarrow \infty}} ~} \frac{1}{mn}|\{ (j,k):j\leq m;k \leq n, (j,k) \in A \}$ if exists.
\end{defn}
The set of $I$-statistical cluster point of $x$ is denoted by ${\Lambda}_x^S(I) $.
\begin{lem}\cite{Ma4}
 Let $ x=\left\{x_{jk}\right\}_{j,k\in\mathbb N}\in\mathbb {R}^n$ be $I$-statistically bounded double sequence. Then for every $\varepsilon>0$ the set 
 \begin{center}
$\left\{(j,k):d({\Lambda}_x^S(I),x_{jk})\geq\varepsilon\right\}$ 
 \end{center}
has $I$-asymptotic density zero, where $d({\Lambda}_x^S(I),x_{jk})=inf_{y\in{\Lambda}_x^S(I)}\left\|y-x_{jk}\right\|$, the distance from $x_{jk}$ to the set ${\Lambda}_x^S(I)$.
\end{lem}

\begin{thm}
For any arbitrary $ c \in {\Lambda}_x^S(I) $ of a double sequence $ x=\left\{x_{jk}\right\}_{j,k\in\mathbb N}$ we have $ \| \xi - c \| \leq r $ for all $ \xi \in I\mbox{-}st\mbox{-}LIM_x^r $.
\end{thm}

\begin{proof}
Assume that there exists a point $ c \in {\Lambda}_x^S(I) $ and $ \xi \in I\mbox{-}st\mbox{-}LIM^r_x $ such that 
$\| \xi - c \| > r $. Let $ \varepsilon = \frac{\| \xi - c \| - r}{3} $. Then 
\begin{eqnarray}
\{ (j,k) \in \mathbb{N\times N}: \|x_{jk} - \xi \| \geq r + \varepsilon \} \supset \{ (j,k) \in \mathbb{N\times N}: \|x_{jk} - c \| < \varepsilon \}. 
\end{eqnarray}
Since $ c \in {\Lambda}_x^S(I) $ we have $ d_I(\{(j,k):\|x_{jk} - c \| < \varepsilon \}) \neq 0 $.
Hence by (3) we have $ d_I(\{(j,k):\|x_{jk} - c \| \geq r + \varepsilon \}) \neq 0 $, which contradicts that 
$ \xi \in I\mbox{-}st\mbox{-}LIM_x^r $. Hence $ \|\xi - c \| \leq r $.
\end{proof}


\begin{thm}
Let $(\mathbb{R}^{n}, \parallel. \parallel)$ be a strictly convex space and $ x = \{x_{jk}\}_{j,k \in \mathbb{N}} $ be a double sequence in this space.
For any $ r > 0 $, let $ y_1, y_2 \in I\mbox{-}st\mbox{-}LIM^r_x $ with $\parallel y_1 - y_2 \parallel = 2r $. Then $x$ is $I$-statistically convergent to $\frac{1}{2}(y_1 + y_2)$.
\end{thm}

\begin{proof}
Let $y_3$ be an arbitrary $I$-statistical cluster point of $x$. Now since $y_1, y_2 \in I\mbox{-}st\mbox{-}LIM^r_x $, so by Theorem 3.9,
\begin{center}
$\parallel y_1 - y_3 \parallel \leq r ~ \mbox {and} ~ \parallel y_2 - y_3 \parallel \leq r $.
\end{center}
Then $ 2r = \parallel y_1 - y_2 \parallel \leq \parallel y_1 - y_3 \parallel + \parallel y_3 - y_2 \parallel \leq 2r $. Therefore $\parallel y_1 - y_3 \parallel = \parallel y_2 - y_3 \parallel = r $. Now
\begin{equation}
\frac{1}{2}(y_1 -y_2) = \frac{1}{2}[(y_3 - y_1) + (y_2 - y_3)]
\end{equation}
Since $\parallel y_1 - y_2 \parallel = 2r $, so $\parallel \frac{1}{2}(y_2 - y_1) \parallel = r $. Again since the space is strictly convex so by (4) we get, $\frac{1}{2}(y_2 - y_1) = y_3 -y_1 = y_2 - y_3 $. Thus $ y_3 = \frac{1}{2}(y_1 + y_2)$ is the unique $I$-statistical cluster point of the double sequence $x$. Again by the given condition $I\mbox{-}st\mbox{-}LIM_x^r  \neq \emptyset $ and so by Theorem 3.3, $x$ is $I$-statistically bounded. Since $y_3$ is the unique $I$-statistical cluster point of the $I$-statistically bounded double sequence $x$, so by Lemma 3.8, $x$ is $I$-statistically convergent to $ y_3 = \frac{1}{2}(y_1 + y_2)$.
\end{proof}


\begin{thm}
Let $ x =\{x_{jk} \}_{j,k \in \mathbb{N}}$ be a double sequence in $X$. 

(i) If $ c \in {{\Lambda}_x^S}(I)$, then $ I\mbox{-}st\mbox{-}LIM_x^r \subset \overline{{B_r}(c)} $.

(ii) $ I\mbox{-}st\mbox{-}LIM_x^r = \underset{ c \in
{\Lambda}_x^S(I) }{\bigcap}{\overline{{B_r}(c)}} = \{\xi \in X:
{\Lambda}_x^S(I)  \subset \overline{B_r(\xi)}\}$.

\end{thm}

\begin{proof}
(i)Let $ c \in {\Lambda}_x^S(I) $. Then by Theorem 3.9, for all $
\xi \in I\mbox{-}st\mbox{-}LIM_x^r $, $ \|\xi - c \| \leq r $ and
hence the result follows.

(ii) By (i) it is clear that $ I\mbox{-}st\mbox{-}LIM_r^x \subset
\underset{ c \in {\Lambda}_x^S(I) }{\bigcap}{\overline{{B_r}(c)}}
$. Now for all $ c \in {\Lambda}_x^S(I) $ and $ y \in \underset{ c
\in {\Lambda}_x^S(I) }{\bigcap}{\overline{{B_r}(c)}}$ we have $\|y
- c \| \leq r $. Then clearly $ \underset{ c \in {\Lambda}_x^S(I)
} {\bigcap}{\overline{{B_r}(c)}} \subset \{ \xi \in X:
{\Lambda}_x^S(I) \subset \overline{{B_r}(\xi)}\} $.

Now, let $ y \notin I\mbox{-}st\mbox{-}LIM_x^r $. Then there
exists an $ \varepsilon > 0 $ such that $ {d_I}(A) \neq 0 $, where
$ A = \{ (m,n) \in \mathbb{N\times N}: \|x_{mn} - y \| \geq r + \varepsilon \} $
. This implies the existence of an $I$-statistical cluster point
$c$ of the sequence $ x $ with $ \|y - c \| \geq r + \varepsilon
$. This gives ${\Lambda}_x^S(I) \nsubseteq \overline{{B_r}(y)}$and
so $ y \notin \{ \xi \in X: {\Lambda}_x^S(I) \subset
\overline{{B_r}(\xi)} \} $. Hence $ \{ \xi \in X: {\Lambda}_x^S(I)
\subset \overline{{B_r}(\xi)} \} \subset I
\mbox{-}st\mbox{-}LIM_x^r $. This completes the proof.

\end{proof}


\noindent\textbf{Acknowledgement:}
The second author is grateful to Government of India for his fellowship funding under UGC-JRF scheme during the preparation of this paper.
\\


\end{document}